\documentclass{article}
\PassOptionsToPackage{numbers,sort&compress}{natbib}
\usepackage[preprint]{neurips_2026}
\usepackage[T1]{fontenc}
\usepackage[utf8]{inputenc}
\usepackage{amsmath,amssymb,amsthm,booktabs,tabularx,microtype,xcolor}
\usepackage[hidelinks]{hyperref}
\hypersetup{pdfauthor={Yukai Song},pdftitle={Sparse coloring for isosceles-free grids: an AI-assisted proof case study},pdfsubject={Preprint; accepted for poster presentation at the NeurIPS 2026 MATH-AI Workshop},pdfkeywords={isosceles-free sets, sparse hypergraph coloring, AI-assisted mathematics}}
\newtheorem{theorem}{Theorem}
\newtheorem{proposition}[theorem]{Proposition}

\newcommand{\Z}{\mathbb Z}
\newcommand{\codeg}{\Delta_{3,2}}
\title{Sparse coloring for isosceles-free grids:\\an AI-assisted proof case study}
\author{Yukai Song \\
University of Pittsburgh \\
\texttt{yus190@pitt.edu}}
\begin{document}
\maketitle
\begin{abstract}
We study the largest size $C(n)$ of a subset of the $n\times n$ integer grid containing no isosceles triple, including equally spaced collinear triples. We prove $C(n)=\Omega\!\left(n\sqrt{\log\log n/\log n}\right)$ by applying the sparse hypergraph coloring theorem of Cooper and Mubayi. Elementary geometry and primitive-direction counts give maximum degree $O(n^2\log n)$ and pair-codegree at most $5n$ for the forbidden-triple hypergraph. The resulting bound improves the explicit guarantee in PatternBoost by a factor of order $\sqrt{\log\log n}$. Recovered interaction records document how a Codex literature proposal, human route selection, finite verification, and author-relayed review informed proof synthesis and repair. We connect these actions to successive proof artifacts and extract two case-derived checking practices: tracking objects, parameters and conclusions in theorem applications, and seeding a definitional omission to test a finite verifier. Separately implemented enumerators agree on complete edge sets for every $2\le n\le12$; suppressing the degenerate-case branch loses exactly the equally spaced collinear triples. These records document assistance under human direction, and finite checks support implementation consistency rather than the asymptotic theorem itself.
\end{abstract}

\section{Problem, result, and scope}
Let $G_n=\{0,\ldots,n-1\}^2$. A set $S\subseteq G_n$ is \emph{isosceles-free} if
\begin{equation}\label{eq:definition}
 \|a-b\|_2\ne\|b-c\|_2\qquad\text{for all pairwise distinct }a,b,c\in S.
\end{equation}
Thus equally spaced collinear triples are forbidden; we call these triples \emph{flat}. Equal distances between disjoint pairs are allowed. Write $C(n)$ for the largest size of such a set. Translation identifies this convention with the problem on $[n]^2$ in PatternBoost~\citep[Section 4.1]{patternboost} and Problem 6.59 of~\citet{discovery}.

\begin{theorem}\label{thm:main}
There exist absolute constants $c>0$ and $n_0$ such that for every integer $n\ge n_0$,
\begin{equation}\label{eq:bound}
C(n)\ge c n\sqrt{\frac{\log\log n}{\log n}}.
\end{equation}
\end{theorem}
All logarithms are natural. PatternBoost~\citep[Proposition A.1]{patternboost} gives $C(n)=\Omega(n/\sqrt{\log n})$ by random selection and deletion. Equation~\eqref{eq:bound} improves that stated guarantee by an unbounded factor of order $\sqrt{\log\log n}$; it does not determine the order of $C(n)$. Our contribution is an application of an existing coloring theorem, not a new general theorem on hypergraphs.

As surveyed by \citet[Section 2]{janosik2026}, this possibly-flat problem is attributed to Wu and Ellenberg--Jain, and also to Erd\H{o}s. Their February 2026 survey records the sublinear PatternBoost lower bound and the conjecture $C(n)=\Theta(n)$.

The coloring mechanism has geometric precedents. Ghosal, Goenka and Keevash~\citep[Section 1.1]{ghosal} record the same logarithmic improvement for sets with no three collinear points; Clemen~\citep{clemen} applies sparse coloring to distinct-slope selection. Neither restriction excludes all noncollinear isosceles triples. The work here is to verify the hypotheses for the isosceles hypergraph.

On the complementary upper-bound side, \citet[Theorem 1.4]{croot2026} prove $C(n)\ll n^2\exp(-c_1\log n/\log\log n)$ for an absolute $c_1>0$, with degenerate triples also forbidden.

Our contributions are the complete application below and an evidence-based account of its AI-assisted development. An earlier internal working draft contained the result; the present workshop version simplifies the proof and its dependencies. The case study links recorded proposals, checks and human instructions to actual proof edits, while separating documented assistance from unobserved attribution or approval.

\section{A complete proof from degree and codegree bounds}
Let $H_n$ be the simple $3$-uniform hypergraph with vertex set $G_n$ and one unordered edge for every triple violating~\eqref{eq:definition}. Write $\alpha(H)$ for the largest independent-set size and $\chi(H)$ for the minimum number of colors in a vertex coloring with no monochromatic edge. A set is independent exactly when it contains no forbidden triple, so $\alpha(H_n)=C(n)$ and $|V(H_n)|=n^2$. For any $3$-uniform hypergraph $H$, let $\Delta(H)$ be its maximum vertex degree and let
\[
\codeg(H)=\max_{u\ne v}|\{e\in E(H):\{u,v\}\subseteq e\}|
\]
be its maximum pair-codegree.

\begin{proposition}\label{prop:counts}
For an absolute constant $B>0$ and every integer $n\ge2$,
$\Delta(H_n)\le Bn^2\log n$ and $\codeg(H_n)\le5n$.
\end{proposition}
\begin{proof}
Fix distinct $u,v$. The third point of an edge containing them lies on one of two circles, centered at $u$ and $v$ with radius $\|u-v\|_2$, or on their perpendicular bisector. Each circle has at most $2n$ grid points, since each first coordinate permits at most two second coordinates. A line has at most $n$ grid points, including a vertical line. A union bound gives $5n$; excluded endpoints or overlaps can only reduce this number.

For the degree bound, fix a vertex $v$. We count separately its roles as apex and base endpoint, allowing overcounting. As apex, the other points are $v+p,v+q$ with $p\ne q$ and $\|p\|_2=\|q\|_2$. Put $U=p-q$ and $W=p+q$. Then $U\ne0$, $U\cdot W=0$, and $\|U\|_\infty,\|W\|_\infty\le2n$. Write $U=gh$, where $g$ is a positive integer and $h=(h_1,h_2)\in\Z^2$ is primitive, meaning $\gcd(|h_1|,|h_2|)=1$. Put $s=\|h\|_\infty$ and $h^\perp=(-h_2,h_1)$. Every integral vector orthogonal to $h$ equals $kh^\perp$ for an integer $k$, whence
\[
1\le s\le2n,\qquad 1\le g\le2n/s,\qquad |k|\le2n/s.
\]
There are at most $8s$ primitive vectors on the sup-norm shell $s$. The identities $p=(W+U)/2$, $q=(W-U)/2$ recover the ordered pair uniquely. Dropping parity and boundary restrictions bounds its count by
\begin{equation}\label{eq:apex}
\sum_{s=1}^{2n}8s\,\frac{2n}{s}\left(\frac{4n}{s}+1\right)
=O(n^2\log n).
\end{equation}
The value $k=0$ is included: it covers antipodal collinear triples with midpoint $v$.

As a base endpoint, the other base endpoint is $b=v+gh$ with $h$ primitive, $s=\|h\|_\infty\le n-1$, and $1\le g\le(n-1)/s$. An apex $a$ must satisfy
\[
2(a-v)\cdot h=g\|h\|_2^2.
\]
If this has integer solutions, they form one translate of $\Z h^\perp$. A coordinate of $h^\perp$ has absolute value $s$, so at most $1+\lfloor(n-1)/s\rfloor\le2n/s$ solutions lie in $G_n$. This gives at most
\begin{equation}\label{eq:base}
\sum_{s=1}^{n-1}(8s)\frac{n}{s}\frac{2n}{s}=O(n^2\log n)
\end{equation}
choices. Every incident edge is covered by at least one role. Both estimates are uniform in $v$, including boundary vertices, and establish the maximum-degree bound.
\end{proof}

\paragraph{The published coloring input.}
The $k=3$ case of Cooper--Mubayi~\citep[Theorem 5]{cooper2016} states that a $3$-uniform hypergraph of maximum degree $\Delta$ with $\codeg(H)\le\sqrt{\Delta}/f$ has
\begin{equation}\label{eq:cm}
\chi(H)=O\!\left(\sqrt{\frac{\Delta}{\log f}}\right).
\end{equation}
We use this in the regime $\Delta\to\infty$, $f\to\infty$, $f=\Delta^{o(1)}$. This regime is permitted: the partition lemma in that paper assumes the upper restriction $f=\Delta^{O(1)}$, and its Section 5.2 handles slow growth. No triangle-free, regularity, or connectedness hypothesis is required here. In particular, this is not a random-greedy guarantee.

\begin{proof}[Proof of Theorem~\ref{thm:main}]
Set $D_n=\lceil Bn^2\log n\rceil$ and $f_n=\sqrt{D_n}/(5n)$. Padding avoids requiring a matching lower bound on $\Delta(H_n)$, while allowing~\eqref{eq:cm} to be applied at the known scale $D_n$. Adjoin to $H_n$ a disjoint $3$-uniform star with exactly $D_n$ edges $\{z,a_i,b_i\}$, all leaf vertices distinct. The resulting hypergraph $\widehat H_n$ has maximum degree exactly $D_n$ and pair-codegree at most $\max\{\codeg(H_n),1\}\le5n=\sqrt{D_n}/f_n$.
Now $D_n=\Theta(n^2\log n)$, $f_n=\Theta(\sqrt{\log n})\to\infty$, and
\[
\log f_n=\tfrac12\log\log n+O(1).
\]
Apply~\eqref{eq:cm} to $\widehat H_n$ and restrict its coloring to $H_n$. Thus
\[
\chi(H_n)\le\chi(\widehat H_n)=O\!\left(n\sqrt{\frac{\log n}{\log\log n}}\right).
\]
A largest color class among the \emph{original} $n^2$ grid vertices has size at least $n^2/\chi(H_n)$ and is independent. This proves~\eqref{eq:bound} with absolute constants for all sufficiently large $n$.
\end{proof}

\section{An evidence-based AI-assisted proof case study}
We use contemporaneous user messages, visible Codex tool calls and outputs, and LaTeX patches. The E1--E3 records are dated 26--27 August 2026 (UTC); their turn metadata reports \texttt{gpt-5.6-sol}, with effort \texttt{max} except for the turn applying the R6 source patch, which records \texttt{ultra}. The starting context included an existing research repository, conditional proofs and random-greedy/ODE candidates; primitive-direction and codegree material was visible before R1. The run did not start from a bare problem statement. We selected the three episodes retrospectively around theorem connection, definition checking and claim repair, using traceable messages, execution results and manuscript changes, rather than sampling all interactions.

Table~\ref{tab:episodes} summarizes the episodes; Appendix~\ref{app:provenance}, including background entry R0, distinguishes original excerpts, translations and summaries. The 7 September submission preparation and current consistency checks used \texttt{gpt-6-astra}, effort \texttt{max}, separately from the August research record. These strings are recorded runtime identifiers, not evidence of model capability. Case-attributable token/time costs, the native Claude model/version and complete context were not recovered. The checking practices below are lessons from this single case.

\begin{table}[ht]
\caption{Recovered development episodes, with evidence IDs defined in Appendix~\ref{app:provenance}. Author-relayed review is distinguished from directly recorded tool use and edits.}
\label{tab:episodes}
\centering
\begin{tabularx}{\linewidth}{@{}>{\raggedright\arraybackslash}p{0.18\linewidth}>{\raggedright\arraybackslash}X>{\raggedright\arraybackslash}X@{}}
\toprule
\textbf{Episode} & \textbf{Proposal and check} & \textbf{Observed consequence} \\
\midrule
E1: theorem connection & Codex retrieved Cooper--Mubayi and proposed its application, separating existence from random-greedy behavior (R1). & The human explicitly authorized replacing the earlier candidate; Codex wrote the core proof artifact (R2--R3). \\
\addlinespace[3pt]
E2: definition check & Historical checks used radius-based generation and direct squared-distance enumeration at selected grid sizes (R4--R5). & At $n=12$ both reported 28,140 edges, 2,520 flat triples, maximum degree 850 and codegree 20. \\
\addlinespace[3pt]
E3: claim repair & Relayed feedback identified an original-graph parameter ceiling that did not cover the padded application (R6). & The historical repair used $H_n$ and its matching degree lower bound (R7). This version retains padding and drops the ceiling claim; R7 also records a separate citation edit. \\
\bottomrule
\end{tabularx}
\end{table}

\paragraph{Object--parameter tracking.}
The historical note used its matching degree and codegree estimates for $H_n$ to assert, in its notation $\Delta_2=\codeg$,
\[
f\le\frac{\sqrt{\Delta(H_n)}}{\Delta_2(H_n)}=\Theta(\sqrt{\log n}),
\]
and presented this as a limitation of its applications (R6--R7). Yet its existence proof applied the theorem to a padded graph of maximum degree $D$, using $f=\sqrt D/(5n)$. The original-graph ceiling did not cover that auxiliary-graph application: the object and its degree parameter had changed. This was a gap in the limitation argument, not an invalid padding construction.

The historical repair moved the existing matching maximum-degree lower bound before the application, then applied the theorem directly to $H_n$ with $f=\sqrt{\Delta(H_n)}/(5n)$. The current proof instead fixes $D_n=\lceil Bn^2\log n\rceil$, retains padding to avoid that lower-bound dependency, and drops the ceiling claim. It therefore makes no claim to bound every admissible auxiliary parameter. Increasing $D$ also changes the numerator of~\eqref{eq:cm}; a larger $f$ alone does not imply a stronger bound.

The case suggests recording five items for each theorem application: the object, each parameter's definition as an actual value or an upper bound, where the hypotheses are checked, the direct conclusion, and its transfer to the target. Here, the theorem bounds $\chi(\widehat H_n)$; restriction bounds $\chi(H_n)$, and a largest color class on the original vertices bounds $\alpha(H_n)$. This check also applies to claimed limitations: restricting a coloring to $H_n$ does not justify transferring a parameter limitation between the two graphs.

A separate R7 edit in the earlier note's linear-hypergraph comparison replaced an independence-number citation with the source's chromatic-number statement and the ensuing largest-color-class deduction. This suggests checking the source's direct conclusion separately from the inference used in the application.

\paragraph{Seeded-defect verification.}
The current consistency checks, prepared on 7 September, cover every $2\le n\le12$: separately implemented squared-distance enumeration and primitive-vector generation produce the same complete forbidden-edge set $E_{\mathrm{correct}}$. As a targeted negative control, we disable the generator's $k=0$ branch, obtaining $E_{\mathrm{mutated}}$. With $E_{\mathrm{flat}}$ the collinear triples in the correct set, the checks establish
\begin{equation}\label{eq:mutation}
E_{\mathrm{correct}}\setminus E_{\mathrm{mutated}}=E_{\mathrm{flat}},
\qquad E_{\mathrm{mutated}}\setminus E_{\mathrm{correct}}=\varnothing.
\end{equation}
The missing sets contain eight triples at $n=3$ and 2,520 at $n=12$. This is an intentionally seeded omission in the current consistency checks, not a historical bug or an estimate of error-detection rates. Both implementations were written by the integrating agent; their agreement is not independent external replication and does not prove the asymptotic theorem. Appendix~\ref{app:checks} and the accompanying code give all counts and checks.

\paragraph{Roles and evidence limits.}
The human chose the route and supplied review feedback; Codex retrieved literature, synthesized and edited proof text, and generated checks. Author-relayed model review is participant-reported evidence; the recovered record does not independently authenticate that review or establish external certification. The records support documented assistance in this case; they do not establish priority, autonomous discovery, causal speedup, comparative model advantage, or general reliability.

\section{Discussion and limitations}
The gain comes from the scale separation $\codeg(H_n)=O(n)$ versus $\sqrt{D_n}=\Theta(n\sqrt{\log n})$. The permitted sparsity factor is of order $\sqrt{\log n}$, whose logarithm supplies the additional $\log\log n$ in the coloring bound. This argument requires all the displayed counting steps; a plausible theorem match alone would not establish them.

The result is an existence bound. It supplies no explicit point set, usable numerical constant, efficient construction, or analysis of the usual random-greedy process. We make no priority or best-known claim beyond the named comparison with PatternBoost. The primitive-direction counts and treatment of vertex roles and flat triples are specific to this geometric application. The proposed transferable practices are to track the object, parameters and conclusion of a theorem application, and to test a verifier against a deliberately omitted part of the definition. This case illustrates these practices; their effectiveness across other problems, broader reliability and efficiency remain unmeasured.

\bibliographystyle{plainnat}
\bibliography{references}

\begin{thebibliography}{7}
\providecommand{\natexlab}[1]{#1}
\providecommand{\url}[1]{\texttt{#1}}
\expandafter\ifx\csname urlstyle\endcsname\relax
  \providecommand{\doi}[1]{doi: #1}\else
  \providecommand{\doi}{doi: \begingroup \urlstyle{rm}\Url}\fi

\bibitem[Charton et~al.(2024)Charton, Ellenberg, Wagner, and
  Williamson]{patternboost}
Fran\c{c}ois Charton, Jordan~S. Ellenberg, Adam~Zsolt Wagner, and Geordie
  Williamson.
\newblock {PatternBoost}: Constructions in mathematics with a little help from
  {AI}.
\newblock \emph{arXiv preprint arXiv:2411.00566v1}, 2024.
\newblock URL \url{https://arxiv.org/abs/2411.00566v1}.

\bibitem[Clemen(2024)]{clemen}
F.~C. Clemen.
\newblock Applications of sparse hypergraph colorings.
\newblock \emph{arXiv preprint arXiv:2406.01499}, 2024.
\newblock URL \url{https://arxiv.org/abs/2406.01499}.

\bibitem[Cooper and Mubayi(2016)]{cooper2016}
Jeff Cooper and Dhruv Mubayi.
\newblock Coloring sparse hypergraphs.
\newblock \emph{SIAM Journal on Discrete Mathematics}, 30\penalty0
  (2):\penalty0 1165--1180, 2016.
\newblock \doi{10.1137/140995805}.
\newblock URL \url{https://arxiv.org/abs/1404.2895}.

\bibitem[Croot et~al.(2026)Croot, Mao, Pohoata, Sheffer, and Yip]{croot2026}
Ernie Croot, Junzhe Mao, Cosmin Pohoata, Adam Sheffer, and Chi~Hoi Yip.
\newblock A combinatorial large sieve for {Sidon} sets, distances, and norm
  forms, 2026.
\newblock URL \url{https://arxiv.org/abs/2606.17487v2}.
\newblock arXiv:2606.17487v2, revised June 24, 2026.

\bibitem[Georgiev et~al.(2025)Georgiev, G\'omez-Serrano, Tao, and
  Wagner]{discovery}
Bogdan Georgiev, Javier G\'omez-Serrano, Terence Tao, and Adam~Zsolt Wagner.
\newblock Mathematical exploration and discovery at scale.
\newblock \emph{arXiv preprint arXiv:2511.02864v3}, 2025.
\newblock URL \url{https://arxiv.org/abs/2511.02864v3}.

\bibitem[Ghosal et~al.(2026)Ghosal, Goenka, and Keevash]{ghosal}
A.~Ghosal, R.~Goenka, and P.~Keevash.
\newblock On subsets of lattice cubes avoiding affine and spherical
  degeneracies.
\newblock \emph{Discrete \& Computational Geometry}, 2026.
\newblock \doi{10.1007/s00454-026-00853-7}.
\newblock URL \url{https://arxiv.org/abs/2509.06935}.
\newblock Published online July 16, 2026.

\bibitem[J{\'a}nosik et~al.(2026)J{\'a}nosik, N{\'a}dor, Nagy, and
  Simon]{janosik2026}
M{\'a}t{\'e} J{\'a}nosik, Art{\'u}r N{\'a}dor, Zolt{\'a}n~L{\'o}r{\'a}nt Nagy,
  and L{\'a}szl{\'o}~Bence Simon.
\newblock Avoiding configurations of small size in the square grid, 2026.
\newblock URL \url{https://arxiv.org/abs/2601.14465v2}.
\newblock arXiv:2601.14465v2, revised February 19, 2026.

\end{thebibliography}
\clearpage
\appendix
\section{Evidence for the process account}\label{app:provenance}
The evidence below was recovered from the contemporaneous conversation and associated proof artifacts. IDs are local to this supplement. Private paths, account identifiers and unrelated content are omitted. Original English/TeX excerpts, English translations of Chinese messages, and editorial summaries are explicitly distinguished. Dates and times are UTC. We do not infer internal reasoning, unlogged actions or complete intellectual provenance from these records.

\paragraph{R0: starting context, runtime metadata and record selection (summary).}
Before R1, the visible task was to assess existing candidates in a research repository for a complete, publishable argument. It included conditional proofs and random-greedy/ODE candidates. Recorded reads already exposed primitive-direction counts and pair-codegree arguments; some earlier material used a nondegenerate-only hypergraph, so its definition was not identical to~\eqref{eq:definition}. R1 was therefore a literature proposal within an existing research effort, not a response to an otherwise empty problem prompt. The visible additional actions were retrieval of Cooper--Mubayi, proposal of its application, subsequent human route selection, creation of a core proof artifact and revisions. Creating that artifact does not identify the first origin of each mathematical idea.

Turn metadata for the cited 26--27 August research events records \texttt{gpt-5.6-sol}. The turn containing the R6 source patch at 04:24 on 27 August records effort \texttt{ultra}; the other cited research events are in turns recorded as \texttt{max}. The separate 7 September submission-preparation record reports \texttt{gpt-6-astra}, effort \texttt{max}; the current consistency checks in Appendix~\ref{app:checks} belong to that preparation stage. These are the enclosing session's logged turn settings; they do not establish separately invoked agents' configurations or support a controlled model comparison. We retrospectively chose E1--E3 to expose a theorem connection, a definition check and a claim repair for which messages, executed checks and manuscript edits could be linked. They are not a random or exhaustive sample. Reliable case-level token/time costs were not recovered: totals for the longer conversation cannot be attributed to these episodes. Native Claude model/version information and its full review context were not recovered either. The surviving excerpts do not reconstruct the complete preceding context.

\paragraph{R1: source retrieval and a visible proposal.}
The recorded Codex tools searched for Cooper--Mubayi, opened \url{https://arxiv.org/pdf/1404.2895}, and received the theorem text. The subsequent public message stated, in translation: ``The sparse hypergraph coloring theorem of Cooper--Mubayi may already give a stronger existence lower bound ... this does not prove that the random greedy algorithm reaches that scale.'' A later message stated the candidate bound, named the degree and pair-codegree obligations, and withheld a priority claim. This supports a visible literature-driven proposal within the recovered record; it does not establish who first made the connection outside it.

\paragraph{R2--R3: human route choice and an actual artifact.}
The user's next instruction read, in translation: ``I agree to replace the ODE candidate with this coloring-theorem application, continue checking novelty, and then deliver the short core paper if that check passes.'' A recorded file patch subsequently created a LaTeX core manuscript with the forbidden-triple definition, the displayed lower bound, primitive-direction degree estimates, and the published coloring input. The instruction is direct evidence of human route selection. The patch is evidence of AI-assisted proof synthesis, not an independent correctness certificate or final human acceptance.

\paragraph{R4--R5: historical checks and their actual scope.}
A recorded Python program grouped grid points by squared radius about each apex and generated unordered edges for $n\in\{2,3,4,5,8,12\}$. A later JavaScript program instead enumerated unordered triples and compared their three squared distances for $n\in\{2,3,4,5,6,8,10,12\}$. Both outputs agree on their common cases; at $n=12$ they report 28,140 edges, 2,520 flat triples, maximum vertex degree 850 and maximum pair-codegree 20. The recovered historical execution does \emph{not} cover every integer from 2 through 12. It also does not compute $C(12)$ or test the asymptotic theorem.

\noindent\begin{minipage}{\linewidth}
\paragraph{R6: the old limitation and relayed feedback.}
An actual source patch at 04:24 on 27 August contains the following original English/TeX excerpt. The historical reference identifier and section number are retained literally; they do not refer to numbering in this paper. In that note, $\Delta_2(H_n)$ denotes the maximum pair-codegree.
\begin{verbatim}
With these tight global parameters, the hypothesis of
Theorem~\ref{thm:CM} forces
\[
 f\le\frac{\Delta(H_n)^{1/2}}{\Delta_2(H_n)}
   =\Theta(\sqrt{\log n}),
\]
and the application in Section~3 attains this order.
\end{verbatim}
\end{minipage}

The same patch went on to state, verbatim: ``This is a limitation of these two applications, not a claim that no future general theorem could extract more information from $H_n$.'' These statements used the original graph's tight parameters. The problem was the extension of that restriction to a proof actually using a padded graph of degree $D$; the inequality for the original graph was not itself the object mismatch.

At 13:50 on 27 August, a user message relayed feedback contrasting those objects and requested a direct application to $H_n$. A short translated excerpt reads: ``Section 4 says [the theorem's hypothesis forces the original-graph bound on $f$], but Section 3 actually applies to $H_n^+$ (maximum degree $D$); its hypothesis allows $f=\sqrt D/(5n)$.'' Brackets replace the formula already displayed above; this is an edited translation, not an independently recovered native review.

\paragraph{R7: the historical repair and the current proof (summary).}
Recorded patches at 13:54 and 13:55 moved the existing matching degree lower-bound proposition ahead of the coloring step and replaced the padded application by an application directly to $H_n$. With $\Delta(H_n)=\Theta(n^2\log n)$ established there, the choice $f=\sqrt{\Delta(H_n)}/(5n)$ supports the claimed coloring bound on that same object. The selected literal patch lines below show the change. The present workshop proof instead uses padding at the fixed scale $D_n$, omits the degree lower-bound dependency and makes no parameter-ceiling claim. These are different proof choices with different claims; the old feedback did not invalidate padding.

A separate patch at 13:59 corrected the earlier note's linear-hypergraph comparison: it replaced a cited independence-number assertion with the source's chromatic-number statement and then deduced the independent-set scale using a largest color class. This citation correction is distinct from the parameter repair. At 15:19, the user reported that a model review had passed the repaired version. The record authenticates the user message and the edits, not the native review, its model settings or its independence.

\paragraph{Selected literal patch lines (R7).}
The recorded edit defined $\Delta=\Delta(H_n)$ immediately before these lines. The signs below are from the actual patch; the ellipsis marks omitted lines. This excerpt documents the historical change, not the proof chosen in this version.
\begin{verbatim}
- D=\lceil K n^2\log n\rceil,
- \qquad f=\frac{\sqrt D}{5n}.
+ f:=\frac{\Delta^{1/2}}{5n}.
...
- \chi(H_n)\le\chi(H_n^+)
- =O\!\left(\sqrt{\frac{D}{\log f}}\right)
+ \chi(H_n)=O\!\left(\sqrt{\frac{\Delta}{\log f}}\right)
\end{verbatim}

\clearpage
\section{Current consistency checks}\label{app:checks}
The accompanying files are in the arXiv ancillary directory \texttt{anc/}. The checker \texttt{verify\_grid.py} uses only the Python standard library. From that directory, run
\begin{quote}
\verb|python verify_grid.py --min-n 2 --max-n 12 --output run.json|
\end{quote}
The two implementations share a coordinate indexing convention but not the edge predicate: one compares three integer squared distances; the other generates $(U,W)=(gh,kh^\perp)$ and enforces parity and square boundaries. Their complete unordered edge sets must agree. Both implementations were written by the integrating agent during the 7 September submission preparation and rerun unchanged for camera-ready validation. These current consistency checks cover every integer from 2 through 12, unlike the historical checks at selected sizes in R4--R5. This is an implementation cross-check, not an independent human or external replication.

\begin{table}[ht]
\centering
\caption{Current consistency checks for forbidden triples. The entries count edges, maximum degrees, pair-codegrees and flat (equally spaced collinear) edges; they do not give optimal isosceles-free set sizes $C(n)$. Complete edge-set equality and the seeded-defect checks hold for each row.}
\label{tab:finite-checks}
\begin{tabular}{rrrrr}
\toprule
$n$ & $|E(H_n)|$ & $\Delta(H_n)$ & $\codeg(H_n)$ & Flat edges \\
\midrule
2 & 4 & 3 & 2 & 0 \\
3 & 44 & 20 & 5 & 8 \\
4 & 172 & 42 & 10 & 24 \\
5 & 516 & 100 & 11 & 72 \\
6 & 1208 & 137 & 12 & 144 \\
7 & 2488 & 224 & 14 & 288 \\
8 & 4504 & 300 & 14 & 480 \\
9 & 7776 & 448 & 15 & 800 \\
10 & 12484 & 545 & 18 & 1200 \\
11 & 19196 & 732 & 19 & 1800 \\
12 & 28140 & 850 & 20 & 2520 \\
\bottomrule
\end{tabular}
\end{table}
The counts satisfy $\codeg(H_n)\le5n$. The targeted negative control suppresses $k=0$ only in the primitive generator. The checker forms the collinear forbidden-edge set directly from squared-distance enumeration, asserts that it equals the full missing-edge set, and asserts that the mutated generator adds no extra edge, as in~\eqref{eq:mutation}. The JSON records these set checks and digests. This diagnoses one deliberately seeded definitional omission; the earlier internal working draft already retains $k=0$. It is not a historical bug or a measurement of sensitivity to other defects.

The checked degree upper expression is $64n^2 H_{2n}+32n^2+16n^2H_{n-1}$, where $H_m=\sum_{s=1}^m1/s$. It follows by expanding the two sums in the main proof. This loose expression is a consistency bound, not a fitted growth law. The result JSON records the actual Python version, generation time, edge-set SHA-256 digests and elapsed times. Timing has no role in the mathematical or AI-effectiveness claims. No maximum-independent-set search was run.

\end{document}